\documentclass[11pt,a4paper]{article}
\usepackage{tsemlines}
\usepackage{tikz} 
\usepackage{xcolor}
\usepackage[colorlinks]{hyperref} 
\usepackage[centertags]{amsmath}
\usepackage{amsfonts}
\usepackage{graphics}
\usepackage{graphicx}
\usepackage{amssymb}
\usepackage{amsthm}
\usepackage{newlfont}
\usepackage{epsfig}
\usepackage[final]{pdfpages}
\newenvironment{wlist}
{\vspace{-10pt}
\begin{list}{}
{\setlength{\labelwidth}{15mm}
\setlength{\partopsep}{0pt}
\setlength{\parskip}{0pt}
\setlength{\topsep}{0pt}
\setlength{\itemsep}{0pt}
\setlength{\parsep}{0pt}
\setlength{\labelsep}{10pt}
\setlength{\leftmargin}{15mm}}
\item[]}
{\end{list}
\vspace{2pt}
\smallskip}

\newtheoremstyle{theorem}
  {10pt}          
  {10pt}  
  {\sl}  
  {\parindent}     
  {\bf}  
  {. }    
  { }    
  {}     
\theoremstyle{theorem}
\newtheorem{theorem}{Theorem}[section]
\newtheorem{corollary}[theorem]{Corollary}

\newtheoremstyle{defi}
  {8pt}          
  {8pt}  
  {\rm}  
  {\parindent}     
  {\bf}  
  {. }    
  { }    
  {}     
\theoremstyle{defi}

\usepackage{latexsym}
\usepackage{color}
\usepackage{amssymb}
\usepackage{amsmath}
\usepackage{amsthm} \usepackage{graphicx}

\newtheorem{problem}{Problem}
\newtheorem{lemma}[theorem]{Lemma}

\date{}
\begin{document}
\title{Wiener index and Steiner 3-Wiener index of a graph}
\maketitle
\vspace{-20mm}
\begin{center}
\author{Matja\v{z} Kov\v{s}e,   
  \and
    Rasila V A,
    \and
    	Ambat Vijayakumar
 }
	\end{center}
\newcommand{\Addresses}{
  \bigskip
  \footnotesize
  
Matja\v{z} Kov\v{s}e, School of Basic Sciences,
IIT Bhubaneswar,
India, \texttt{matjaz.kovse@gmail.com}
\par\nopagebreak
  Rasila V A,
  Department of Mathematics, Cochin University of Science and Technology, India, \texttt{17rasila17@gmail.com} \par\nopagebreak  
  Ambat Vijayakumar,
	 Department of Mathematics, Cochin University of Science and Technology, India, \texttt{vambat@gmail.com}}

\begin{abstract}
Let $S$ be a set of vertices of a connected graph $G$. The Steiner distance of $S$ is the minimum size of a connected subgraph of $G$ containing all the vertices of $S$. The sum of all Steiner distances on sets of size $k$ is called the Steiner $k$-Wiener index, hence for $k=2$ we get the Wiener index. The modular graphs are graphs in which every three vertices $x, y$ and $z$ have at least one median vertex $m(x,y,z)$ that belongs to shortest paths between each pair of $x, y$ and $z$. The Steiner 3-Wiener index of a modular graph is expressed in terms of its Wiener index. As a corollary formulae for the Steiner 3-Wiener index of Fibonacci and Lucas cubes are obtained.
\end{abstract}

\noindent
{\bf MR Subject Classifications:} 05C12
\bigskip\noindent


\section{Introduction}
All graphs in this paper are simple, finite and undirected.  
If $G$ is a connected graph
and $u, v \in V (G)$, then the {\em (geodetic) distance} $d_G(u, v)$ (or simply $d(u, v)$ if it is clear we are dealing with $G$) between $u$ and $v$ is the
number of edges on a shortest path connecting $u$ and $v$. The {\em Wiener index $W (G)$} of a connected graph $G$ is defined by
$$W (G) =\sum_ {\{u,v\} \in V(G)} d(u,v).$$
It has been introduced in 1947 by Wiener who showed in \cite{wiener1947} that there exist correlations between the boiling points of paraffins and their molecular structure.
The {\em average distance} $\mu(G)$ of a graph $G$ is defined to be the average of all
distances between pairs of vertices in $G$, i.e.
$$
\mu(G)={n\choose 2}^{-1}\sum_{\{u,v\}\subseteq V(G)}d(u,v).
$$
Hence $W (G) = {n\choose 2} \mu(G)$.

In \cite{chartrand1989} Chartrand, Oellermann, Tian  and Zou introduced 
The {\em Steiner distance} of a graph as a natural generalization of the geodetic 
distance. Let $S$ be a set of vertices of a connected graph $G$. The Steiner distance $d(S)$  of $S$ is the minimum size (the number of edges) of a connected subgraph $H$ of $G$ containing all the vertices of $S$. Clearly $H$ is a subtree of $G$, called {\em Steiner tree} connecting vertices of $S$. If $S = \{u, v\}$, then the Steiner distance coincides with the geodetic distance. See the survey on Steiner distance in \cite{mao2017survey} for known results.

The {\em average $k$ Steiner distance} $\mu_k(G)$ of a graph $G$ has been 
introduced by Dankelmann, Oellermann, and Swart in
\cite{dankelmann1996average}, as the average of the Steiner
distances of all subsets of $V(G)$ of size $k$, i.e.
$$
\mu_k(G)={n\choose k}^{-1}\sum_{S\subseteq V(G), \ |S|=k}d(S).
$$

For $2\leq r<k$, Dankelmann, Oellermann, and Swart \cite{dankelmann1996average} established a relation between $\mu_r(G)$, $\mu_{k+1-r}(G)$, and $\mu_k(G)$.
\begin{theorem}{\upshape \cite{dankelmann1996average}}\label{th4-1}
Let $G$ be connected graph and $2\leq r\leq k-1$. Then
$$
\mu_k(G)\leq \mu_r(G)+\mu_{k+1-r}(G).
$$
\indent Moreover for $k\geq 3$, $\mu_k(G)\leq (k-1)\mu(G)$.
\end{theorem}

The bounds in Theorem \ref{th4-1} are sharp for the complete graph.
In \cite{dankelmann1996average} it has been noted that for
each connected graph $G$ of order $n$ and $3\leq k\leq n$,
$$
\mu_k(G)\leq \frac{k+1}{k-1}\mu_{k-1}(G).
$$

Regarding the lower bound for $\mu_k(G)$ in terms of $\mu(G)$, the conjecture has been posed in \cite{dankelmann1996average}  that the smallest ratio $\mu_k(G)/\mu(G)$ taken over all connected graphs $G$ of order $n$ where $n\geq k$, is attained if $G$ is the path. More formally:
If $G$ is a connected graph of order $n$ and $3\leq k\leq n$, then

\begin{equation}
\label{formula:mukmu}
\mu_k(G)\geq 3\frac{k-1}{k+1}\mu(G).
\end{equation}

In \cite{dankelmann1996average} it has been proved that the conjecture
is true for $k=3$ and $k=n$. Later the conjecture has been disproved in  \cite{jiang2006average} by Jiang, who showed that for all positive integers $m\geq2$, we have
\begin{align*}
&f(m)\leq \text{ inf }  \Big\{\dfrac{\mu_{m}(G)}{\mu(G)}: G  \text{ is connected and } |V(G)| \geq m \Big \}
\leq 2-\dfrac{1}{2^{m-2}}
\end{align*}
where
$$ f(m)= \begin{cases}  2-\dfrac{2}{m}   & \text{ for even }m ,\\
 2-\dfrac{2}{m+1}  &  \text{ for odd }m. \end{cases}$$

Hence, in particular,
$$ \displaystyle{ \lim_{m\rightarrow\infty}} \text{ inf }  \Big\{\dfrac{\mu_{m}(G)}{\mu(G)}: G  \text{ is connected and } |V(G)| \geq m \Big \}=2.
$$

In \cite{gutman2016discuss}, Li, Mao and Gutman introduced the Steiner $k$-Wiener index $SW_{k}(G)$ of a connected graph $G$ as $$SW_k(G) =  \sum_ {\substack{S\subseteq V(G)\\ | S| =k}} d(S).$$
For $k = 2$, the Steiner $k$-Wiener index coincides with the Wiener index. The average $k$-Steiner distance $\mu_k (G)$ is related to the $k$-Steiner Wiener index via the equality $\mu_k (G) = SW_{k}(G)/ {n \choose k}$. In \cite{gutman2016discuss} the formulae for the exact values of the $k$-Steiner Wiener index of several simple families of graphs have been derived together with sharp lower and upper bounds for general graphs and trees. Moreover it has been shown that the Steiner $3$-Wiener index of a tree of order $n$ is directly related to the ordinary Wiener index as follows.

\begin{theorem}\label{thm:sw3trees}
\begin{equation}
SW_3(T) = \frac{n-2}{2}\,W(T).
\end{equation}
\end{theorem}

The paper is organised as follows. In Section \ref{section:3SWmodular} we generalise Theorem \ref{thm:sw3trees} to modular graphs, and obtain the inequality holding for general graphs. In Section \ref{section:S3Wproduct} the Steiner 3-Wiener index of Cartesian product of modular graphs is expressed in terms of the Steiner 3-Wiener index of the factor graphs. In Section \ref{section:S3WFibonacci} we derive formulae for the Steiner 3-Wiener index of Fibonacci and Lucas Cubes. In Section \ref{section:3SWblock} we derive a formula for the Steiner 3-Wiener index of block graphs. In the final section we conclude with an open problem on the number of non-modular triplets in a graph $G$ and its relation with the Steiner 3-Wiener index of $G$.

\section{Steiner 3-Wiener index of modular graphs}
\label{section:3SWmodular}

The {\em interval} $I(u,v)$ between two vertices $u$ and $v$ consists of all vertices that are on shortest paths joining $u$ and $v$. A graph $G$ is a {\em modular} \cite{bandelt1987modular} if for every three vertices $x,y,z$ there exists a vertex $w$ that lies on a shortest path between every two vertices of $x, y, z$, i. e. 
\begin{equation}
\label{eq:modular}
|I(x,y)\cap I(x,z) \cap I(y,z)| \geq 1.
\end{equation}
Their name comes from the fact that a finite lattice is a modular lattice if and only if its Hasse diagram is a modular graph. It is easy to see that a modular graph is a bipartite graph. Examples of modular graphs are trees, hypercubes, grids, complete bipartite graphs, etc. The simplest example of non-modular graphs are cycles on $n$ vertices, for $n\neq 4$, and complete graphs.  A graph $G$ is called a {\em median graph} if equality holds in (\ref{eq:modular}) . Hence in a median graph every triple of 
vertices $u,v,w$ has a unique {\em median} - a vertex that simultaneously lies on a shortest $u,v$-path,
a shortest $u,w$-path, and a shortest $v,w$-path. 

For $S \subseteq V(G)$, the {\em 2-intersection interval} of $S$ is the intersection of all intervals between pairs of vertices from $S$: $$\displaystyle I_2(S)=\bigcap_{\substack{a,b \in S\\ a\neq b}}I(a,b).$$ Hence modular graphs are those graphs for which the 2-intersection interval of every triple of vertices is non-empty. The following result is from \cite{kubicka1998steiner}.

\begin{theorem}\label{thm:2intersection}
Let $S = \{u_1, u_2,\ldots,u_n\}$ be a set of $n > 2$ vertices of a graph $G$. If the 2-intersection interval of $S$ is nonempty and $x \in I_2(S)$, then $$d(S) = \sum_{i=1}^{n}d(u_i, x).$$ 
\end{theorem}

Let $G$ be a connected graph. A triplet of vertices $x,y,z \in V(G)$ is called a {\em modular triplet} if $I(x,y) \cap I(x,z) \cap I (y,z) \neq \emptyset$.  
Next we provide the connection between 3-Steiner Wiener index and Wiener index of a graph.

\begin{theorem}
\label{thm:sw3modular}
Let $G$ be a graph on $n$ vertices. Then,
$$
SW_3(G) \geq \frac{n-2}{2} W(G),
$$
with the equality if and only if $G$ is a modular graph.
\end{theorem}

\begin{proof}
Let $S=\{a,b,c\} \subseteq V(G)$, $|S|=3,$ and let $G$ be a modular graph. Then there exist $x \in I_2 (S)$. By Theorem \ref{thm:2intersection} it follows that $d(S) = d(a, x)+d(b, x)+d(c,x)$. There are two possibilities: $x \in S$ or $x \notin S$.\\
{\bf Case 1.} $x \in S$\\
Without loss of generality, let $x=b$. Hence $d(a,c)=d(a,b)+d(b,c)$ and therefore $d(S)=d(a,c)=\frac 12 (d(a,b)+d(b,c)+d(a,c))$.\\
{\bf Case 2.} $x \notin S$\\
It follows that $d(a,x)+d(x,b)=d(a,b)$ and $d(b,x)+d(x,c)=d(b,c)$ and $d(a,x)+d(x,c)=d(a,c)$. Therefore $d(S)=d(a, x)+d(b, x)+d(c,x)=\frac 12 (d(a,b)+d(b,c)+d(a,c))$.
Each pair of vertices in a graph on $n$ vertices belongs to $n-2$ different triples of vertices, hence it follows.
\begin{align*}
SW_3(G) &= \sum_{\substack{S\subseteq V(G)\\ |S|=3}}d(S)\\
&=  \sum_{\substack{a,b,c \in V(G)\\ |\{a,b,c\}|=3}} \frac 12 \left(d(a,b)+d(b,c)+d(a,c)\right)\\
&= \frac 12 \sum_{a,b \in V(G) }d(a,b) \, (n-2)\\ \label{line 3}
&=  \frac{n-2}{2} W(G).
\end{align*}
For a non modular triplet $a,b,c \in V(G)$ we always have $d(\{a,b,c\}) > \frac 12 (d(a,b) + d(a,c) + d(b,c))$, hence for a non modular graph
\begin{align*}
 \frac{n-2}{2} W(G) =\sum_{\substack{a,b,c \in V(G)\\ |\{a,b,c\}|=3}} \frac 12 \left(d(a,b)+d(b,c)+d(a,c)\right)< \sum_{\substack{S\subseteq V(G)\\ |S|=3}}d(S) = SW_3(G).
\end{align*}
\end{proof}

Since trees are modular graphs Theorem \ref{thm:sw3modular} generalises Theorem \ref{thm:sw3trees}. Moreover the inequality for general graphs coincides with the inequality \ref{formula:mukmu} for $k=3$,  proved for $k$ by Dankelmann, Oellermann, and Swart in \cite{dankelmann1996average}. 

\section{Steiner 3-Wiener index of Cartesian products of modular graphs}
\label{section:S3Wproduct}

The {\em Cartesian product} $G \,\square\, H$ of two graphs $G$ and $H$ is the graph with vertex set $V(G)\times V(H)$ and $(a,x)(b,y)\in E(G\,\square\, H)$ whenever either $ab\in E(G)$ and $x=y$, or $a=b$ and $xy\in E(H)$. The following result has been obtained by Graovac and Pisanski \cite{graovac1991} and Yeh and Gutman \cite{YehGutman}.

\begin{theorem}
\label{lem:wienerproduct}
Let G and H be connected graphs. Then
$$W(G\square H)=|V(G)|^2 W(H)+|V(H)|^2 W(G).$$
\end{theorem}

For $k \geq 3$ the situation is much more complicated.  In \cite{MaoWangGutman}, Mao Wang and Gutman obtained the following bounds.

\begin{theorem}\label{th4-16}
Let $G$ be a connected graph with $n$ vertices, and let $H$ be a
connected graph with $m$ vertices. Let $k$ be an integer with $2\leq
k\leq nm$. Then
\begin{eqnarray*}
&&\sum_{x=2}^{k}\binom{m}{r_1}\binom{m}{r_2}\cdots \binom{m}{r_{x}}SW_{x}(G)
+ \sum_{y=2}^{k}\binom{n}{s_1}\binom{n}{s_2}\cdots
\binom{n}{s_{y}}SW_{y}(G) \\[3mm]
&\leq &SW_k(G\Box H)
\leq \frac{k}{2}\left[\sum_{x=2}^{k}\binom{m}{r_1}\binom{m}{r_2}
\cdots \binom{m}{r_{x}}SW_{x}(G) +
\sum_{x=2}^{k}\binom{n}{s_1}\binom{n}{s_2}\cdots \binom{n}{s_{y}}SW_{y}(G)\right]
\end{eqnarray*}
where $\sum_{i=1}^{x}r_{i}=k$ and $r_i\geq 1$, and
$\sum_{i=1}^{y}s_{i}=k$ and $s_i\geq 1$.
\end{theorem}

Since interval of Cartesian product $G\square H$ equals the Cartesian product of the corresponding intervals in $G$ and $H$, it follows that the Cartesian product of two modular graphs is a modular graph.

 \begin{theorem} Let $G$ and $H$ be modular graphs. Then,
\begin{equation*}
SW_3(G\square H) = (|V(G)|\cdot|V(H)|-2) \left( \frac{|V(G)|^2 }{|V(H)|-2} SW_3(H)+\frac{|V(H)|^2}{|V(G)|-2} SW_3(G) \right)
 \end{equation*}
\end{theorem}

\begin{proof}
Since $|G\square H|=|V(G)|\cdot|V(H)|$ by Theorem \ref{thm:sw3modular} and Theorem \ref{lem:wienerproduct} it follows that
\begin{align*}
   SW_3(G\square H) &= \frac{|V(G)|\cdot|V(H)|-2}{2}  W(G\square H) \\
   &= \frac{|V(G)|\cdot|V(H)|-2}{2} \left( |V(G)|^2 W(H)+|V(H)|^2 W(G) \right)\\
   &= \frac{|V(G)|\cdot|V(H)|-2}{2} \left( |V(G)|^2 \frac{2}{|V(H)|-2} SW_3(H)+|V(H)|^2 \frac{2}{|V(G)|-2} SW_3(G) \right)\\
   &= (|V(G)|\cdot|V(H)|-2) \left( \frac{|V(G)|^2 }{|V(H)|-2} SW_3(H)+\frac{|V(H)|^2}{|V(G)|-2} SW_3(G) \right).
\end{align*}
\end{proof}

Note that the simplest way to calculate Steiner 3-Wiener index of Cartesian product of modular graphs is by applying the equality $$SW_3(G\square H) =\frac{|V(G)|\cdot|V(H)|-2}{2} \left( |V(G)|^2 W(H)+|V(H)|^2 W(G)\right).$$

\section{Steiner 3-Wiener index of Fibonacci and Lucas Cubes}
\label{section:S3WFibonacci}

Fibonacci cubes were introduced as a model for interconnection networks 
in \cite{hsu-93a,hsu-93b}.
Lucas cubes were introduced in \cite{muci-01} for similar reason. They have been studied extensively afterwards, see survey \cite{klavzar2013survey}.

The Fibonacci numbers $F_n$ are defined as
$F_1 = F_2 = 1$, $F_{n} = F_{n-1} + F_{n-2}$ for $n\geq 3$, 
and the Lucas numbers as
$L_0 = 2$, $L_1 = 1$, $L_{n} = L_{n-1} + L_{n-2}$ for $n\geq 2$. Lucas numbers are related to Fibonacci number by the identities $L_{n}=F_{n-1}+F_{n+1}=F_{n}+2F_{n-1}=F_{n+2}-F_{n-2}$.

The vertex set of the {\em $n$-cube} $Q_n$ consists of all binary 
strings of length $n$, two vertices being adjacent if the 
corresponding strings differ in precisely one place. 
A {\em Fibonacci string} of length $n$ is a binary string
$b_1b_2\ldots b_n$ with $b_ib_{i+1}=0$ for $1\leq i<n$, that is, 
a binary string without two consecutive ones. The {\em Fibonacci cube} 
$\Gamma_n$ ($n\geq 1$) is the subgraph of $Q_n$ induced by 
the Fibonacci strings of length $n$. For convenience we also set 
$\Gamma_0 = K_1$. Call a Fibonacci string $b_1b_2\ldots b_n$ a 
{\em Lucas string} if $b_1b_n = 0$. Then the {\em Lucas cube} 
$\Lambda_n$ ($n\geq 2$) is the subgraph of $Q_n$ induced by 
the Lucas strings of length $n$. We also set $\Lambda_0 = K_1$. For the Fibonacci cubes we have $|V(\Gamma_n)|=F_{n+2}$, see \cite{hsu-93b}, and for the Lucas cubes we have $|V(\Lambda_n)| = L_{n}$, see \cite{muci-01}. 

The Fibonacci cube $\Gamma_1$ is isomorphic to the path graph $P_2$ on two vertices,
$\Gamma_2$ to the path graph $P_3$, 
$\Gamma_3$ to the banner graph - 4-cycle and path graph $P_2$  joined in a common vertex.
The Lucas cube $\Lambda_2$ is isomorphic to the path graph $P_3$, $\Lambda_3$ to the star on four vertices $S_4$, and $\Lambda_4$ to two 4-cycles joined in a common vertex.
The Fibonacci cubes $\Gamma_4$ and $\Gamma_5$ and the Lucas cube 
$\Lambda_5$ and their corresponding binary labels of vertices are shown in Figure \ref{slika-1}. 

\begin{figure}[htb]
\begin{center}
\unitlength=1.6mm
\linethickness{0.4pt}
\begin{picture}(69.00,74.00)
\put(2.00,48.00){\circle*{1.50}}
\put(12.00,48.00){\circle*{1.50}}
\put(22.00,48.00){\circle*{1.50}}
\put(2.00,58.00){\circle*{1.50}}
\put(12.00,58.00){\circle*{1.50}}
\put(22.00,58.00){\circle*{1.50}}
\put(2.00,68.00){\circle*{1.50}}
\put(12.00,68.00){\circle*{1.50}}
\emline{2.00}{68.00}{1}{2.00}{48.00}{2}
\emline{2.00}{48.00}{3}{22.00}{48.00}{4}
\emline{22.00}{48.00}{5}{22.00}{58.00}{6}
\emline{22.00}{58.00}{7}{2.00}{58.00}{8}
\emline{2.00}{68.00}{9}{12.00}{68.00}{10}
\emline{12.00}{68.00}{11}{12.00}{48.00}{12}
\put(42.00,48.00){\circle*{1.50}}
\put(52.00,48.00){\circle*{1.50}}
\put(62.00,48.00){\circle*{1.50}}
\put(42.00,58.00){\circle*{1.50}}
\put(52.00,58.00){\circle*{1.50}}
\put(62.00,58.00){\circle*{1.50}}
\put(42.00,68.00){\circle*{1.50}}
\put(52.00,68.00){\circle*{1.50}}
\emline{42.00}{68.00}{13}{42.00}{48.00}{14}
\emline{42.00}{48.00}{15}{62.00}{48.00}{16}
\emline{62.00}{48.00}{17}{62.00}{58.00}{18}
\emline{62.00}{58.00}{19}{42.00}{58.00}{20}
\emline{42.00}{68.00}{21}{52.00}{68.00}{22}
\emline{52.00}{68.00}{23}{52.00}{48.00}{24}
\emline{42.00}{58.00}{25}{47.00}{63.00}{26}
\emline{47.00}{63.00}{27}{47.00}{73.00}{28}
\emline{47.00}{73.00}{29}{57.00}{73.00}{30}
\emline{57.00}{73.00}{31}{57.00}{63.00}{32}
\emline{57.00}{63.00}{33}{52.00}{58.00}{34}
\emline{42.00}{68.00}{35}{47.00}{73.00}{36}
\emline{47.00}{63.00}{37}{57.00}{63.00}{38}
\emline{52.00}{68.00}{39}{57.00}{73.00}{40}
\emline{57.00}{63.00}{41}{67.00}{63.00}{42}
\emline{67.00}{63.00}{43}{62.00}{58.00}{44}
\put(67.00,63.00){\circle*{1.50}}
\put(57.00,63.00){\circle*{1.50}}
\put(57.00,73.00){\circle*{1.50}}
\put(47.00,73.00){\circle*{1.50}}
\put(1.00,48.00){\makebox(0,0)[rc]{1001}}
\put(1.00,58.00){\makebox(0,0)[rc]{0001}}
\put(1.00,68.00){\makebox(0,0)[rc]{0101}}
\put(40.00,48.00){\makebox(0,0)[rc]{01001}}
\put(40.00,58.00){\makebox(0,0)[rc]{00001}}
\put(40.00,68.00){\makebox(0,0)[rc]{00101}}
\put(53.50,70.00){\makebox(0,0)[rc]{00100}}
\put(14.00,68.00){\makebox(0,0)[lc]{0100}}
\put(14.00,60.00){\makebox(0,0)[lb]{0000}}
\put(24.00,58.00){\makebox(0,0)[lc]{0010}}
\put(24.00,48.00){\makebox(0,0)[lc]{1010}}
\put(14.00,50.00){\makebox(0,0)[lb]{1000}}
\put(12.00,43.00){\makebox(0,0)[ct]{$\Gamma_4$}}
\put(52.00,43.00){\makebox(0,0)[ct]{$\Gamma_5$}}
\put(45.00,73.00){\makebox(0,0)[rc]{10101}}
\put(59.00,73.00){\makebox(0,0)[lc]{10100}}
\put(58.00,64.00){\makebox(0,0)[lb]{10000}}
\put(69.00,63.00){\makebox(0,0)[lc]{10010}}
\put(64.00,58.00){\makebox(0,0)[lc]{00010}}
\put(64.00,48.00){\makebox(0,0)[lc]{01010}}
\put(53.00,49.00){\makebox(0,0)[lb]{01000}}
\put(53.00,57.00){\makebox(0,0)[lt]{00000}}
\put(46.00,61.00){\makebox(0,0)[lt]{10001}}
\put(47.00,63.00){\circle*{1.50}}
\put(2.00,8.00){\circle*{1.50}}
\put(12.00,8.00){\circle*{1.50}}
\put(22.00,8.00){\circle*{1.50}}
\put(2.00,18.00){\circle*{1.50}}
\put(12.00,18.00){\circle*{1.50}}
\put(22.00,18.00){\circle*{1.50}}
\put(2.00,28.00){\circle*{1.50}}
\put(12.00,28.00){\circle*{1.50}}
\emline{2.00}{28.00}{45}{2.00}{8.00}{46}
\emline{2.00}{8.00}{47}{22.00}{8.00}{48}
\emline{22.00}{8.00}{49}{22.00}{18.00}{50}
\emline{22.00}{18.00}{51}{2.00}{18.00}{52}
\emline{2.00}{28.00}{53}{12.00}{28.00}{54}
\emline{12.00}{28.00}{55}{12.00}{8.00}{56}
\put(1.00,8.00){\makebox(0,0)[rc]{01001}}
\put(1.00,18.00){\makebox(0,0)[rc]{00001}}
\put(1.00,28.00){\makebox(0,0)[rc]{00101}}
\put(5.00,30.00){\makebox(0,0)[lc]{00100}}
\put(15.00,19.00){\makebox(0,0)[lb]{00000}}
\put(24.00,18.00){\makebox(0,0)[lc]{00010}}
\put(24.00,8.00){\makebox(0,0)[lc]{01010}}
\put(13.00,9.00){\makebox(0,0)[lb]{01000}}
\put(12.00,3.00){\makebox(0,0)[ct]{$\Lambda_5$}} 
\end{picture}
\caption{$\Gamma_4$, $\Gamma_5$ and $\Lambda_5$}
\label{slika-1}
\end{center}
\end{figure}
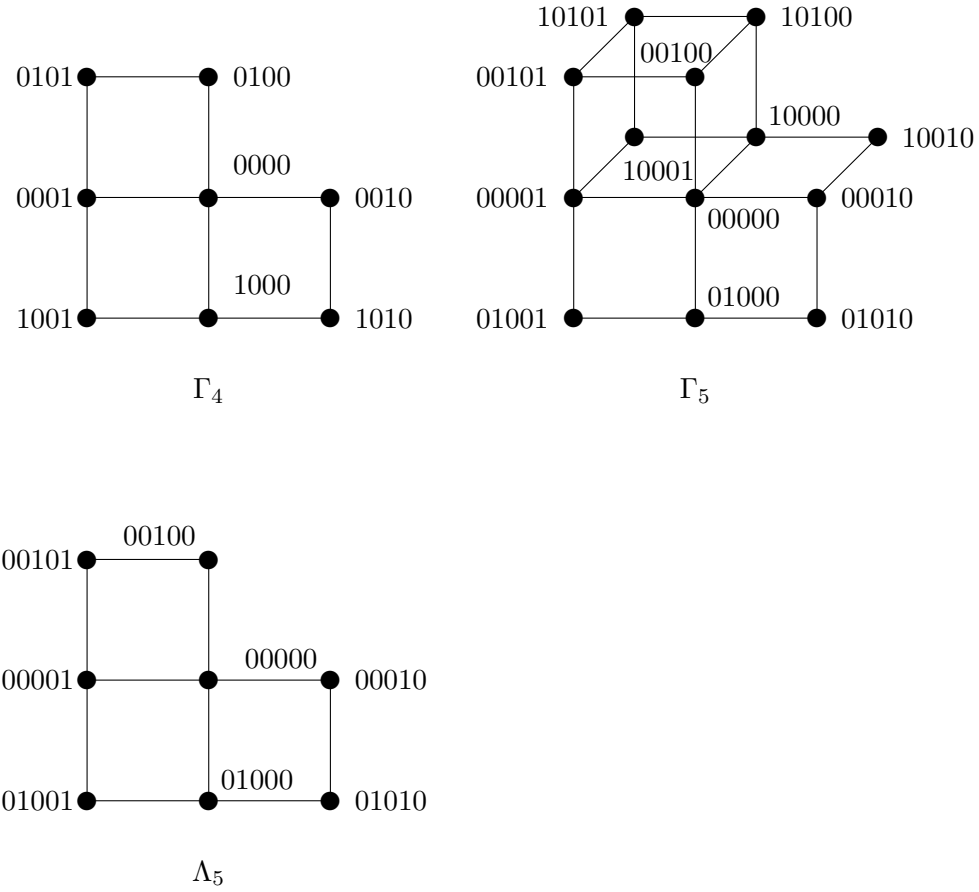
\vspace{5mm}
We sumarize the main results from \cite{klavzar2012fibo} in the following theorem.

\begin{theorem}
\label{thm:Wfb}
For Fibonacci cube $\Gamma_n$ and Lucas cube $\Lambda_n$ the following holds.
\begin{wlist}
\item[(i)] For any $n \geq 0$,  ${\displaystyle W(\Gamma_n) = \frac{1}{25}\left(4(n + 1)F_n^2 + (9n + 2)F_nF_{n+1} + 6nF_{n+1}^2\right)}$,
\item[(ii)] For any $n \geq 1$, ${\displaystyle W(\Lambda_n) = n F_{n-1} F_{n+1}}$,
\item[(iii)] ${\displaystyle \lim_{n\to \infty } \frac{\mu(\Gamma_n)}{n}}= {\displaystyle \lim_{n\to \infty } \frac{\mu(\Lambda_n)}{n} = \frac{2}{5}}$.
\end{wlist}
\end{theorem}

In \cite{klavzar2005fibomedian} Klav\v{z}ar has shown that Fibonacci and Lucas cubes are median graphs. Hence they are modular graphs and we can apply Theorem \ref{thm:sw3modular} and Theorem \ref{thm:Wfb} to obtain formulae for the Steiner 3-Wiener index of Fibonacci and Lucas cubes.

\begin{corollary} 
Let $k\geq 2$ be an integer. 
\begin{wlist}
\item[(i)] For any $n \geq 0$, $SW_3(\Gamma_n) = \frac{1}{50}(F_{n+2} -2)\left(4(n + 1)F_n^2 + (9n + 2)F_nF_{n+1} + 6nF_{n+1}^2\right)$,
\item[(ii)] For any $n \geq 1$, $SW_3(\Lambda_n) = \frac{n}{2} F_{n-1} F_{n+1} (L_{n}-2) $.
\end{wlist}
\end{corollary}

The first values of the sequence $\{SW_3(\Gamma_n)\}$ are:\\
 $0, 0, 2, 24, 162, 968, 5206, 26672, 131652, 634752, 3006708, \ldots $\\
 
\indent The first values of the sequence $\{SW_3(\Lambda_n)\}$ are:\\
 $0, 0, 2, 9, 100, 540, 3120, 15876, 79560, 384615, 1830730 \ldots$

\begin{corollary} 
Let $k$ be an integer with $2\leq k$. Then
 $${\displaystyle \lim_{n\to \infty } \frac{\mu_k(\Gamma_n)}{n}}= {\displaystyle \lim_{n\to \infty } \frac{\mu_k(\Lambda_n)}{n} = \frac 35}.$$
\end{corollary}

\begin{proof}
From Binet’s formula for the Fibonacci numbers, see \cite{wiki1}, it follows that
 \begin{equation*}
\label{eq2}
{\displaystyle \lim_{n\to \infty }{\frac {F_{n+k}}{F_{n}}}=\varphi^{k}={\frac {1+{\sqrt {5}}}{2}}}.
\end{equation*}

Then (i) and the fact that $|V(\Gamma_n)|=F_{n+2}$ imply:

\begin{align*}
{\displaystyle \lim_{n\to \infty } \frac{\mu_k(\Gamma_n)}{n}}&=\lim_{n \to \infty } \frac{ \left(F_{n+2} -2)(4(n + 1)F_n^2 + (9n + 2)F_nF_{n+1}+ 6nF_{n+1}^2\right) }{50 n} \binom{F_{n+2}}{3}^{-1} \\
&=\lim_{n \to \infty } \frac{ 3\left(4(n + 1)F_n^2 + (9n + 2)F_nF_{n+1}+ 6nF_{n+1}^2\right) } 
{25 n(F_{n+2} -3)(F_{n+2}-1))}\\
&= 3 \left( \frac{4}{25\varphi^{4}}+\frac{9}{25\varphi^{3}}+ \frac{6}{25\varphi^{2}} \right)\\
&= \frac 35
\end{align*}


See \cite{wiki2} for the following equality: $L_{n}^{2}=5F_{n}^{2}+4(-1)^{n}.$
Then (ii) and the fact that $|V(\Lambda_n)| = L_{n}$ imply:
\begin{align*}
\lim_{n\to \infty } \frac{\mu_k(\Lambda_n)}{n} &= 
{\displaystyle \lim_{n\to \infty }  \frac{n F_{n-1} F_{n+1} (L_{n}-2) }{2 n} \binom{L_{n}}{3}^{-1}}\\
&={\displaystyle \lim_{n\to \infty }  \frac{ 3 F_{n-1} F_{n+1}}{(L_{n})(L_{n}-1)}}\\
&={\displaystyle \lim_{n\to \infty }  \frac{ 3 F_{n-1} F_{n+1}}{5F_{n}^{2}+4(-1)^{n}}}\\
&= \frac 35
\end{align*}
\end{proof}

\section{Steiner 3-Wiener index of block graphs}
\label{section:3SWblock}


A block of a graph is a maximal connected vertex induced subgraph that has no cut vertices. A {\em block graph} is a graph in which every block is a clique. In \cite{yeh2008centers} Steiner distance related subsets of vertices of block graphs has been studied, followed by the study of Steiner $k$-Wiener index of block graphs in \cite{kovse2018block}.

A claw-free graph is a graph in which no induced subgraph is a claw, i.e. a complete bipartite graph $K_{1,3}.$ Claw-free block graphs are block graphs which are claw-free. They are equivalent to the line graphs of trees.

For a graph $G$, let $n(G)$ denote the number of its vertices. For a graph $G$ with $p$, $p \geq 3$, connected components $G_1, G_2, \ldots, G_p$ let
$$
N_3(G)=\sum_{\substack{\{i,j,k\} \subseteq \{1,2, \ldots, p\}}} n(G_i) \cdot n(G_j)   \cdot n(G_k)    
$$

For a graph $G$ with $p$, $p < 3$, connected components we set  $N_3(G)=0$. Note that $N_3(G)$ counts the number of triplets of vertices belonging to three different connected components of $G$.

Let $G$ be a block graph with blocks $B_1, B_2,\ldots, B_t$, and let $G \setminus B_i$ denote a graph obtained from $G$ by deleting all edges from block $B_i$. Let $nm(G)$ denote the number of non-modular triples of graph $G$. 

\begin{lemma}\label{lem:non-modular}
Let $G$ be a block graph with blocks $B_1, B_2, \ldots, B_t$.  Then 
$$nm(G)=   \sum_{i=1}^t N_3(G \setminus B_i).$$
\end{lemma}

\begin{proof}
Let $a,b,c \in V(G)$ be a non-modular triplet. Let $L(a,b,c)$ denote a subgraph of $G$ induced by all blocks $B_i$ with $B_i \cap (I(a,b)\cup I(a,c) \cup I(b,c)) \neq \emptyset$. We distinguish three cases.\\

{\bf Case 1.} $a,b$ and $c$ belong to the same block  $B_i$ of $G$. 

Hence $a,b$ and $c$ belong to three different connected components of $G\setminus (B_i)$ and for any other block $B_j$ they belong to the same connected component of $G\setminus (B_j)$.\\
 
{\bf Case 2.} Exactly two vertices of the triplet $a,b$ and $c$ belong to the same block.\\
W. l. o. g. let $a,b \in B_i$. Since $a \notin I(b,c)$ and $b \notin I(a,c)$ there exist a vertex $d \in B_i$ such that $d \in I(b,c) \cap I(a,c)$. Hence $a,b$ and $c$ belong to three different connected components of $G \setminus (B_i)$ and for any other block $B_j$ they belong to at most two connected components of $G\setminus (B_j)$.\\

{\bf Case 3.} $a,b$ and $c$ belong to three different blocks of $G$: $B_i,B_j$ and $B_k$.\\
Since $I(a,b)\cap I(a,c) \cap I(b,c)) = \emptyset$ it follows that $L(a,b,c)$ is a claw free subgraph of $G$. Hence for exactly one of the blocks $B_i,B_j$ and $B_k$ it holds that it has a nonempty intersection with the shortest path between the vertices from the remaining two blocks. W. l. o. g. let this be block $B_k$, where $c \in B_k$. Hence $a,b$ and $c$ belong to three different connected components of $G \setminus (B_k)$ and for any other block $B_j$ they belong to at most two connected components of $G\setminus(B_j)$.
\end{proof}

From the proof of Lemma \ref{lem:non-modular} it follows that block graphs are pseudo-median graphs: for every three vertices, either there exists a unique vertex that belongs to shortest paths between all three vertices, or there exists a unique triangle whose edges lie on these three shortest paths.

\begin{theorem}\label{thm:SW3}
Let $G$ be a block graph with blocks $B_1, B_2, \ldots, B_m$ and $|V(G)|=n$. Then $$SW_3(G) =  \frac{n-2}{2} W(G) + \frac 12 \sum_{i=1}^t N_3(G \setminus B_i).$$
\end{theorem}

\begin{proof}Let $M(G)$ denote the set of all modular triplets of $G$ and $NM(G)$ the set of all non-modular triplets. From the proof of Lemma \ref{lem:non-modular} it follows that for $x,y, z \in NM(G)$ we have $$d(\{x,y,z\}) =  \frac 12 (d(x,y) + d(x,z) + d(y,z)) + \frac 12.$$ Using the same argument as in the proof of Theorem \ref{thm:sw3modular} and Lemma \ref{lem:non-modular} it follows
\begin{align*}
SW_3(G) &= \sum_{x,y,z\in M(G)} d(\{x,y,z\}) + \sum_{x,y,z\in NM(G)} d(\{x,y,z\}) \\
&= \sum_{x,y,z\in M(G)} \left(\frac 12 (d(x,y) + d(x,z) + d(y,z)\right)\\
 &     +\sum_{x,y,z\in NM(G)} \left(\frac 12 (d(x,y) + d(x,z) + d(y,z)) + \frac 12 \right)\\
 &=  \frac{n-2}{2} W(G) + \frac 12 nm(G)\\
&= \frac{n-2}{2} W(G) + \frac 12  \sum_{i=1}^t N_3(G \setminus B_i).
\end{align*}

\end{proof}

\section{Conclusion}
\label{section:conclusion}

One possible way to compute the Steiner 3-Wiener index of an arbitrary graph $G$ could be to use a similar approach as in Section \ref{section:3SWblock} and by finding the number of non-modular triplets of $G$ establish a relation between the Steiner 3-Wiener index and Wiener index of $G$.

\begin{problem}
Find a simple procedure to compute the number of non-modular triplets in a graph $G$.\end{problem}

\begin{problem}
Establish a relation between the Steiner 3-Wiener index and Wiener index of a graph $G$ belonging to a particular graph family.
\end{problem}

\vspace{-2mm}
\Addresses

\end{document}